\documentclass[12pt]{amsart}

\usepackage{hyperref}

\usepackage{amssymb,latexsym, amsthm, enumerate, mathptmx}

\newtheorem{theorem}{Theorem}[section]
\newtheorem{lemma}[theorem]{Lemma}
\newtheorem{proposition}[theorem]{Proposition}
\newtheorem{corollary}[theorem]{Corollary}
\newtheorem{claim}{Claim}

\newtheoremstyle{definition}
  {6pt}
  {6pt}
  {}
  {}
  {\bfseries}
  {.}
  {.5em}
  {}%

\theoremstyle{definition}

\newtheoremstyle{remark}
  {6pt}
  {6pt}
  {}
  {}
  {\bfseries}
  {.}
  {.5em}
  {}%

\theoremstyle{remark}
\newtheorem{remark}[theorem]{Remark}

\newtheoremstyle{example}
  {6pt}
  {6pt}
  {}
  {}
  {\bfseries}
  {.}
  {.5em}
  {}%

\theoremstyle{example}
\newtheorem{example}[theorem]{Example}

\makeatletter
\renewcommand\@makefntext[1]{%
\setlength\parindent{1em}%
\noindent
\makebox[1.8em][r]{}{#1}}
\makeatother

\DeclareMathOperator{\depth}{depth}
\DeclareMathOperator{\core}{core}
\DeclareMathOperator{\link}{link}
\DeclareMathOperator{\st}{star}
\DeclareMathOperator{\rank}{rank}

\title{On the Gorensteinness of broken circuit complexes and Orlik--Terao ideals}

\author{Dinh Van Le}
\address{Universit\"at Osnabr\"uck, Institut f\"ur Mathematik, 49069 Osnabr\"uck, Germany}
\email{dlevan@uos.de}

\begin{document}

\begin{abstract}
 It is proved that the broken circuit complex of an ordered matroid is Gorenstein if and only if it is a complete intersection. Several characterizations for a matroid that admits such an order are then given, with particular interest in the $h$-vector of broken circuit complexes of the matroid. As an application, we prove that the Orlik--Terao algebra of a hyperplane arrangement is Gorenstein if and only if it is a complete intersection. Interestingly, our result shows that the complete intersection property (and hence the Gorensteinness as well) of the Orlik--Terao algebra can be determined from the last two nonzero entries of its $h$-vector.
\end{abstract}

\maketitle
 \section{Introduction and main results}

The broken circuit complex was introduced for a graph by Wilf \cite{W} based on the idea of Whitney \cite{Wh}. This notion then was extended to matroids by Brylawski \cite{Br} and has been studied by various authors; see, e.g.,  \cite{B,BZ,BrOx}. The broken circuit complex is important because of both of its combinatorial and algebraic aspects. On the one hand, the entries of its $f$-vector coincide with the coefficients of the Poincar\'{e} polynomial of the matroid \cite{B}. On the other hand, the broken circuit complex defines two algebras which are deformations of two important algebras arising in the theory of hyperplane arrangements: the Orlik--Solomon algebra and the Orlik--Terao algebra; see \cite{B,PS}.

A well-known property of the broken circuit complex is that it  is shellable \cite{Pr}. It follows, in particular, that the Stanley--Reisner ring of the broken circuit complex and the Orlik--Terao algebra are Cohen--Macaulay.  Natural questions then arise: when are these algebras complete intersections? or Gorenstein? Characterizations for the complete intersection property of these algebras were obtained in \cite{LR}; see also \cite{DGT} for related results. A partial answer to the Gorensteinness of the broken circuit complex is also given in \cite{LR}, in which it is shown that Gorenstein broken circuit complexes of codimension 3 are complete intersections. However, a complete answer to this question seems, especially for the Orlik--Terao algebra, much more complicated.

\footnotetext{
\begin{itemize}
\item[ ]
{\it Mathematics  Subject  Classification} (2010): 05B35, 05E40, 05E45, 13F55, 13H10, 52C35.
\item[ ]
{\it Key words and phrases}: broken circuit complex, complete intersection, Gorenstein, hyperplane arrangement, matroid, Orlik--Terao algebra.
\end{itemize}
}

The aim of this paper is to investigate the Gorensteinness of the Stanley--Reisner ring of the broken circuit complex and the Orlik--Terao algebra. Quite surprisingly, we are able to show, among other things, that these algebras are Gorenstein exactly when they are complete intersections, thus giving a satisfactory complete answer to the `Gorenstein' question mentioned above. More precisely, for the broken circuit complex, we obtain the following:

\begin{theorem} \label{th11}
Let $\Delta:=BC(M,<)$ be the broken circuit complex of an ordered matroid $(M,<)$. Let $K$ be an arbitrary field. Then the following conditions are equivalent:
\begin{enumerate}
\item
 $\Delta$ is Gorenstein over $K$;
\item
 $\Delta$ is locally Gorenstein over $K$;
\item
 for any face $F$ with $\dim\link_\Delta F=1$, $\link_\Delta F$ is Gorenstein over $K$;
\item
 for any face $F$ with $\dim\link_\Delta F=1$, $\link_\Delta F$ is a complete intersection complex;
\item
 $\Delta$ is a locally complete intersection complex;
\item
$\Delta$ is a complete intersection complex.
\end{enumerate}
\end{theorem}

The Gorensteinness of the Orlik--Terao algebra is equivalent to the symmetry of its $h$-vector \cite[Theorem 4.4]{S2}, and thus equivalent to the symmetry of the $h$-vector of a broken circuit complex associated to that Orlik--Terao algebra. For this latter property, we have the following characterizations:

\begin{theorem}\label{th12}
 Let $M$ be a simple matroid. Let $(h_0,\ldots,h_{s})$ be the $h$-vector of broken circuit complexes of $M$, where $h_{s}\ne 0$. Then the following conditions are equivalent:
\begin{enumerate}
 \item There exists an ordering $<$ of the ground set of $M$ such that the broken circuit complex $BC(M,<)$ is a complete intersection;
 \item the $h$-vector $(h_0,\ldots,h_{s})$ satisfies the Dehn--Sommerville equations: $h_i=h_{s-i}$ for $i=0,\ldots,s$;
 \item $h_0=h_{s}$ and $h_1=h_{s-1}$;
 \item each connected component of $M$ is either a coloop or isomorphic to an iterated parallel connection of uniform matroids of the form $U_{m,m+1}$ ($m\geq 2$).
\end{enumerate}
\end{theorem}

From this, one easily gets:

\begin{theorem}\label{th13}
 Let $\mathcal{A}$ be an essential central hyperplane arrangement. Let $\mathbf{C}(\mathcal{A})$ be the Orlik--Terao algebra of $\mathcal{A}$. Then the following conditions are equivalent:
\begin{enumerate}
 \item $\mathbf{C}(\mathcal{A})$ is a complete intersection;
 \item $\mathbf{C}(\mathcal{A})$ is Gorenstein;
 \item $M(\mathcal{A})$ satisfies one of the equivalent conditions in Theorem \ref{th12}, where $M(\mathcal{A})$ is the underlying matroid of $\mathcal{A}$.
\end{enumerate}
\end{theorem}

Note that for an essential central hyperplane arrangement, the first two entries of the $h$-vector of its Orlik--Terao algebra are fixed: $h_0=1$ and $h_1=n-r$, where $n$ is the number of hyperplanes and $r$ is the rank of the arrangement. Therefore, Theorem \ref{th13} implies an interesting consequence: the Gorensteinness (and hence the complete intersection property) of the Orlik--Terao algebra depends only on the last two nonzero entries of its $h$-vector.

 The paper is divided into 4 sections. In Section 2 we recall some notions and basis facts about the Orlik--Terao algebra, the broken circuit complex, and series-parallel networks. Section 3 contains the proof of Theorem \ref{th11} and related results. Section 4 is devoted to the proofs of Theorem \ref{th12} and Theorem \ref{th13}.

\noindent{\bf Acknowledgments.} The author would like to thank Professor Tim R\"{o}mer for raising his interest in the topic of the paper and for valuable discussions. He also thanks the referees for useful suggestions.

\section{Preliminaries}

 In this section we collect several notions and properties concerning simplicial complexes, the Orlik--Terao algebra, matroids, the broken circuit complex, and series-parallel networks which will be used later. The reader is referred to \cite{OT,Ox,S} for unexplained terminology.

\subsection{Simplicial complexes}

Let $\Delta$ be a simplicial complex with vertices $[n]=\{1,\ldots,n\}$. Let $K$ be a field and $S=K[x_1,\ldots,x_n]$ a polynomial ring over $K$. The \emph{Stanley--Reisner ideal} $I_\Delta$ of $\Delta$ (over $K$) is the ideal in $S$ generated by all monomial $x_{i_1}\cdots x_{i_p}$ such that $\{i_1,\ldots,i_p\}\not\in\Delta.$ The $K$-algebra $K[\Delta]=S/I_\Delta$ is then called the \emph{Stanley--Reisner ring} of $\Delta$. For a face $F\in\Delta$ and a subset $W\subseteq [n]$ we define
$$\begin{aligned}
  \st_\Delta F&=\{G\in\Delta:F\cup G\in\Delta\},\\
\link_\Delta F&=\{G\in\Delta:F\cup G\in\Delta,F\cap G=\emptyset\},\\
\Delta_W&=\{G\in\Delta: G\subseteq W\}.
\end{aligned}$$
Moreover, we put $\core \Delta=\Delta_{\core [n]}$, where $\core [n]=\{i\in [n]: \st_\Delta\{i\}\ne \Delta\}$.

Let us say that $\Delta$ is \emph{Cohen--Macaulay} (resp. \emph{Gorenstein}, a \emph{complete intersection}) over $K$ if so is the ring $K[\Delta]$. We then define $\Delta$ to be \emph{locally Gorenstein} over $K$ (resp. a \emph{locally complete intersection}) if $\link_\Delta\{i\}$ is Gorenstein over $K$ (resp. a complete intersection) for $i=1,\ldots,n$.

Assume that $\dim \Delta=r-1$. Then the \emph{$f$-vector} of $\Delta$ is the $(r+1)$-tuple $(f_0,\ldots,f_{r})$, where $f_i$ is the number of faces of $\Delta$ of cardinality $i$. The $f$-vector is the most natural way to encode the combinatorics of a simplicial complex. However, in many cases it is more convenient to work with a related vector, the \emph{$h$-vector} $(h_0,\ldots,h_{r})$, defined as follows
$$h_i=\sum_{j=0}^i(-1)^{i-j}\binom{r-j}{i-j}f_j,\quad i=0,\ldots,r.$$
Equivalently, if we let $f(t)=\sum_{i=0}^rf_it^{r-i}$ be the $f$-polynomial and $h(t)=\sum_{i=0}^rh_it^{r-i}$ the $h$-polynomial of $\Delta$, then $f(t)=h(t+1)$. In this paper we usually consider the $h$-vector with zero entries at the end removed.

\subsection{The Orlik--Terao algebra}

Let $\mathcal{A}=\{H_1,\ldots,H_n\}$ be an essential central hyperplane arrangement in a vector space $V$ over a field $K$. The Orlik--Terao algebra of $\mathcal{A}$ introduced in \cite{OT2} is a commutative analog of the Orlik--Solomon algebra. Recently, it has gotten more attention because it encodes subtle information which is missed in the Orlik--Solomon algebra; see \cite{DGT, LR, Sc, Sc2, ST,T} for more details.

Let $\alpha_i\in K[V]$ be linear forms such that $\ker \alpha_i=H_i$ for $i=1,\ldots,n$. Then the \emph{Orlik--Terao algebra} $\mathbf{C}(\mathcal{A})$ of $\mathcal{A}$ is the subalgebra of $K(V)$ generated by reciprocals of the linear forms $\alpha_i$, that is, $\mathbf{C}(\mathcal{A})=K[1/\alpha_1,\ldots,1/\alpha_n]$. Let $S={K}[x_1,\ldots,x_n]$ be the polynomial ring in $n$ variables. The kernel of the surjection
$$S\rightarrow\mathbf{C}(\mathcal{A}),\ x_i\mapsto 1/\alpha_i\ \text{ for }\ i=1,\ldots,n$$
is called the \emph{Orlik--Terao ideal} of $\mathcal{A}$, denoted by $I(\mathcal{A})$. Thus, $\mathbf{C}(\mathcal{A})=S/I(\mathcal{A})$.

It is known that $\mathbf{C}(\mathcal{A})$ is a standard graded Cohen--Macaulay domain; see \cite[Proposition 2.1]{ST} and \cite[Theorem 4]{PS}. Therefore, it follows from a result due to Stanley \cite[Theorem 4.4]{S2} that the Gorensteinness of $\mathbf{C}(\mathcal{A})$ depends only on its Hilbert series, or equivalently on its $h$-vector. This allows one to use combinatorial tools to study the Gorensteinness of $\mathbf{C}(\mathcal{A})$: the $h$-vector of $\mathbf{C}(\mathcal{A})$ can be computed via a combinatorial object - the broken circuit complex of the underlying matroid of $\mathcal{A}$. Recall that the \emph{underlying matroid} $M(\mathcal{A})$ of $\mathcal{A}$ is the matroid on the ground set $\mathcal{A}$ whose independent sets are the independent subsets of $\mathcal{A}$. It was proved by Proudfoot and Speyer \cite[Theorem 4]{PS} that the Stanley--Reisner ideal of the broken circuit complex of $M(\mathcal{A})$ with respect to an arbitrary ordering of $\mathcal{A}$ (see below) is an initial
ideal of $I(\mathcal{A})$. Therefore, $\mathbf{C}(\mathcal{A})$ and that complex share the same $h$-vector.

\begin{remark}
 The above relation between the Orlik--Terao ideal $I(\mathcal{A})$ and the Stanley--Reisner ideal of the broken circuit complex of $M(\mathcal{A})$ was exploited in \cite[Proposition 2.4]{LR} to give a short proof for a formula essentially due to Terao \cite[Theorem 1.2]{T} which relates the Hilbert series of $\mathbf{C}(\mathcal{A})$ to the Poincar\'{e} polynomial of $\mathcal{A}$: Let $(f_0,\ldots,f_r)$ be the $f$-vector of the broken circuit complex of $M(\mathcal{A})$. We call $\pi(\mathcal{A},t)=\sum_{p=0}^rf_{i}t^i$ the \emph{Poincar\'{e} polynomial} of $\mathcal{A}$. Then one has the following formula for the Hilbert series of $\mathbf{C}(\mathcal{A})$:
$$H_{\mathbf{C}(\mathcal{A})}(t)=\pi\Big(\mathcal{A},\frac{t}{1-t}\Big).$$
\end{remark}

\subsection{Matroids}

A {\it matroid} $M$ on the ground set $E$ is a collection $\mathcal{I}$ of subsets of $E$, called \emph{independent sets}, satisfying the following conditions:
\begin{enumerate}
\item
$\emptyset\in\mathcal{I};$
\item
If $I\in \mathcal{I}$ and $I'\subseteq I$, then $I'\in\mathcal{I}$;
\item
If $I,I'\in \mathcal{I}$ and $|I'|<| I|$, then there is an element $e\in I-I'$ such that $I'\cup\{e\}\in\mathcal{I}.$
\end{enumerate}
Maximal independent sets of $M$ are called \emph{bases}. They have the same cardinality which is called the \emph{rank} of $M$. \emph{Dependent sets} are subsets of $E$ that are not in $\mathcal{I}$. Minimal dependent sets are called \emph{circuits}. Denote by $\mathcal{C}(M)$ the set of all circuits of $M$. Clearly,  $\mathcal{C}(M)$ determines $M$: $\mathcal{I}$ consists of subsets of $E$ that do not contain any member of $\mathcal{C}(M)$.

Let $e\in M$ (by abuse of notation, we also write $e\in M$ if $e\in E$). Then $e$ is a \emph{loop} if $\{e\}$ is a circuit of $M$. We call $e$ a \emph{coloop} if it is contained in every basis of $M$. Two elements $e,f\in M$ are \emph{parallel} if they form a circuit. A \emph{parallel class} of $M$ is a maximal subset of $E$ whose any two distinct members are parallel and whose no member is a loop. A parallel class is  \emph{non-trivial} if it contains at least two elements. The matroid $M$ is \emph{simple} if it has no loops and no non-trivial parallel classes. In general, one may associate to an arbitrary matroid $M$ a simple matroid $\overline{M}$, called the \emph{simplification} of $M$, by identifying all members of each parallel class and deleting all the loops from $M$.

The \emph{deletion} of $M$ at an element $e$, denoted $M-e$, is the matroid on the ground set $E-e$ whose independent sets are those members of $\mathcal{I}$ which do not contain $e$. The \emph{contraction} of $M$ at $e$, denoted $M/e$, is the matroid whose ground set is also $E-e$ and whose circuits consist of the minimal non-empty members of $\{C-e:C\in\mathcal{C}(M)\}$.

A typical example of a matroid is the underlying matroid of a central hyperplane arrangement introduced above. There are two other common examples which will be concerned in this paper.

\begin{example}\label{ex1}
(i) Let $m\leq n$ be non-negative integers and let $E$ be an $n$-element set. The \emph{uniform matroid} $U_{m,n}$ on $E$ is the matroid whose independent sets are the subsets of $E$ of cardinality at most $m$. This matroid has rank $m$ and its circuits are the $(m+1)$-element subsets of $E$. Thus in particular, $U_{m,n}$ is simple if and only if $m\geq 2$. Moreover, when $n=m+1$, the matroid $U_{m,m+1}$ has a unique circuit which is often denoted by $C_{m+1}$.

(ii) Let $G$ be a graph whose edge set is $E.$ Let $\mathcal{C}$ be the set of edge sets of cycles of $G$. Then $\mathcal{C}$ forms the set of circuits of a matroid $M(G)$ on $E$. We call $M(G)$ the \emph{cycle matroid} (or \emph{polygon matroid}) of $G$. This matroid is simple if and only if $G$ is a simple graph.
\end{example}

Let $M_1$ and $M_2$ be two matroids on disjoint ground sets $E_1$ and $E_2$. We define their \emph{direct sum} $M_1\oplus M_2$ to be the matroid on the ground set $E_1\cup E_2$ whose independent sets are the unions of an independent set of $M_1$ and an independent set of $M_2$. A matroid is called \emph{separable} if it is the direct sum of two smaller matroids. A matroid is \emph{connected} if it is not separable. Every matroid has a unique decomposition as a direct sum $M=M_1\oplus\cdots\oplus M_k$ of connected matroids; we call $M_1,\ldots,M_k$ \emph{connected components} of $M$.

It is apparent from the definition that the collection of independent sets of a matroid $M$ forms a simplicial complex. We call it the \emph{matroid complex} (or \emph{independence complex}) of $M$. When $M$ has rank $r$, this complex is a pure complex of dimension $r-1$.

\subsection{The broken circuit complex}

Let $M$ be a matroid on the ground set $E$. Given a linear order $<$ on $E$, a \emph{broken circuit} is a circuit with its least element removed. For each circuit $C$ of $M$, the broken circuit corresponding to $C$ will be denoted by $bc(C)$. The \emph{broken circuit complex} of $M$ with respect to $<$ is the family of all subsets of $E$ that contain no broken circuit, denoted by $BC(M,<)$ (or $BC(M)$ when the order $<$ is specified). It is a subcomplex of the same dimension as the matroid complex of $M$. Note that the broken circuit complex of $M$ is isomorphic to that of its simplification \cite[Proposition 7.4.1]{B}, so one always may assume that the matroid $M$ is simple when working with its broken circuit complex.

Fix a linear order on $E$ and let $e_0$ be the smallest element of $E$. Then $BC(M)$ is the cone over the \emph{reduced broken circuit complex} $\overline{BC}(M)$ with apex $e_0$. Here, $\overline{BC}(M)$ is the family of all subsets of $E-e_0$ that do not contain any broken circuit. It is well-known that $BC(M)$ and  $\overline{BC}(M)$ are shellable complexes; see \cite{Pr} and also \cite[7.4]{B}. Hence, when $M$ has rank $r$, these complexes are Cohen--Macaulay of dimension $r-1$ and $r-2$, respectively. Clearly, $BC(M)$ is a Gorenstein (resp. complete intersection) complex if and only if so is $\overline{BC}(M)$.

Observe that the broken circuit complex $BC(M,<)$ depends on the given order $<$ in the sense that another order may yield a non-isomorphic broken circuit complex. However, its $h$-vector is independent of the choice of order; see \cite[7.4]{B}. We summarize some properties of the $h$-vector of the broken circuit complex in the next proposition.

\begin{proposition}\label{pr21}
 Assume that $M$ is a simple matroid of rank $r$ on an $n$-element ground set. Let $(h_0,h_1,\ldots,h_r)$ and $h_M(t)=h_0t^r+h_1t^{r-1}+\cdots+h_r$ be respectively the $h$-vector and the $h$-polynomial of a broken circuit complex of $M$. Then the following statements hold:
\begin{enumerate}
 \item $h_0=1,\ h_1=n-r,\ h_{r-1}=\beta(M)$, and $h_r=0$.

 \item If $M=M_1\oplus M_2$, then $h_M(t)=h_{M_1}(t)h_{M_2}(t)$.

 \item $M$ has $k$ connected components if and only if $k$ is the smallest number such that $h_{r-k}\ne 0$. In particular, $\beta(M)>0$ if and only if $M$ is connected.

 \item If $e\in M$ is not a coloop, then $h_M(t)=h_{M-e}(t)+h_{\overline{M/e}}(t)$, where $\overline{M/e}$ is the simplification of $M/e$. In particular, if $M$ is connected then either $M-e$ or $\overline{M/e}$ is connected. Moreover, if $\beta(M)=1$ then exactly one of $M-e$ and $\overline{M/e}$ is connected.

 \item Assume that $M$ is representable. Let $s$ be the largest index such that $h_s\ne 0$. Then $\sum_{j=0}^ih_j\leq\sum_{j=0}^ih_{s-j}$ for all $i=0,\ldots,s$.
\end{enumerate}
\end{proposition}

In the above proposition, $\beta(M)$ is the \emph{beta invariant} of $M$, defined by Crapo \cite{Cr}. The reader is referred to \cite{Za} for more information on this invariant. The case $\beta(M)=1$ will be discussed in the next subsection.

\begin{proof}[Proof of Proposition \ref{pr21}]
 Note that $h_M(t)=T_M(t,0)$, where $T_M(t,u)$ is the Tutte polynomial of $M$ \cite[p. 240]{B}. So (i)-(iv) can be deduced from the corresponding properties of the Tutte polynomial which are contained in \cite[6.2]{BrOx2}. To prove (v), let $\mathcal{A}$ be a central hyperplane arrangement such that $M=M(\mathcal{A})$. Then the $h$-vector of the broken circuit complex of $M$ coincides with the $h$-vector of the Orlik--Terao algebra $\mathbf{C}(\mathcal{A})$ of $\mathcal{A}$. Since $\mathbf{C}(\mathcal{A})$ is a standard graded Cohen--Macaulay domain, the result follows from \cite[Theorem 2.1]{S1}.
\end{proof}

\begin{remark}
 Regarding Proposition \ref{pr21}(v), a major open question in matroid theory asks whether the stronger inequalities $h_i\leq h_{s-i},$ $i=0,\ldots,\lfloor s/2\rfloor$, hold for arbitrary matroids; see \cite{Sw}. As far as we know, the question has not been answered even for the case $h_s=1$. In this context, it is natural to ask whether Proposition \ref{pr21}(v) holds without the assumption on the representability of the matroid. We personally expect an affirmative answer. But it would be much more fascinating if the answer was negative. Because in this case the answer to the above open question would also be negative and Proposition \ref{pr21}(v) would give a non-trivial necessary condition on the Tutte polynomial for the representability of matroids.
\end{remark}

\subsection{Series-parallel networks}

The notion of series-parallel networks has its origin in electrical network theory. We recall here several properties of series and parallel connections. A full treatment of this topic may be found in \cite{Br2} or \cite{Ox}.

Let $M_1,M_2$ be matroids on the ground sets $E_1,E_2$ with $E_1\cap E_2=\{e\}$, where $e$ is neither a loop nor a coloop in $M_1$ or $M_2$. Denote by $\mathcal{C}(M)$ the set of circuits of a matroid $M$. Then the \emph{series connection} $S(M_1,M_2)$ and the \emph{parallel connection} $P(M_1,M_2)$ of $M_1,M_2$ relative to $e$ are the matroids on the ground set $E_1\cup E_2$ whose sets of circuits are respectively:
$$\begin{aligned}
   \mathcal{C}(S(M_1,M_2))&=\mathcal{C}(M_1-e)\cup\mathcal{C}(M_2-e)\cup\{C_1\cup C_2: e\in C_i\in\mathcal{C}(M_i)\ \text{ for }\ i=1,2\},\\
\mathcal{C}(P(M_1,M_2))&=\mathcal{C}(M_1)\cup\mathcal{C}(M_2)\cup\{C_1\cup C_2-e: e\in C_i\in\mathcal{C}(M_i)\ \text{ for }\ i=1,2\}.
  \end{aligned}$$

The following properties of series and parallel connections are given in \cite[Proposition 4.6, Corollary 4.11, Theorem 6.16(v)]{Br2} and \cite[Proposition 2.3]{Br3}.

\begin{proposition}\label{pr22}
 Let $M,M'$ be matroids on the ground sets $E,E'$ with $E\cap E'=\{e\}$, where $e$ is neither a loop nor a coloop in $M$ or $M'$. Let $S(M,M')$ and $P(M,M')$ be the series and parallel connections of $M,M'$ relative to $e$, respectively. Then the following statements hold:
\begin{enumerate}
\item $P(M,M')$ is connected (resp. simple) if and only if both $M$ and $M'$ are connected (resp. simple).

\item $S(M,M')$ is simple if and only if $M$ and $M'$ are loopless, each has no $2$-element circuit which do not contain $e$, and each has at most one $2$-element circuit which contains $e$. In particular, if this is the case, then $S(M,M')$ has at most one $3$-element circuit which contains $e$.

\item $h_{P(M,M')}(t)=t^{-1}h_M(t)h_{M'}(t)$, where $h_M(t)$ is the $h$-polynomial of a broken circuit complex of $M$.

\item Assume $M$ is connected and $f\in M$. Then $M-f$ is separable if and only if $M$ is a series connection relative to $f$; $\overline{M/f}$ is separable if and only if $M$ is a parallel connection relative to $f$, where $\overline{M/f}$ is the simplification of $M/f$.

\item If $M$ is connected, then $M$ can be decomposed as an iterated parallel connection of parallel irreducible connected matroids.
\end{enumerate}
\end{proposition}

A matroid $M$ is a \emph{series-parallel network} if it can be obtained from the $2$-element circuit $C_{2}$ by a sequence of operations each of which is either a series or a parallel connection. Some characterizations of series-parallel networks are provided in the next proposition. For the proof, see \cite[Theorem 7.6]{Br2} and \cite[Theorem 5.4.10, Corollary 11.2.15]{Ox}.

\begin{proposition}\label{pr23}
 Let $M$ be a connected matroid on a ground set of two or more elements. Then the following statements are equivalent:
\begin{enumerate}
 \item $M$ is a series-parallel network;

 \item $\beta(M)=1$;

 \item $M$ is the cycle matroid of a block $G$ having no subgraph that is a subdivision of the complete graph $K_4$.
\end{enumerate}
\end{proposition}

Here by a \emph{block} we mean a connected graph (in the usual sense of graph theory) whose cycle matroid is connected. The graph $G$ in the last statement of the above proposition is also called a (graphical) series-parallel network. Note that such a graph which is simple always contains a vertex of degree 2, by a result due to Dirac \cite{Di}; see also \cite[Lemma 5.4.1]{Ox}.

\section{The Gorensteinness of the broken circuit complex}

In this section we prove Theorem \ref{th11} and discuss some related results. We will need a characterization of Gorenstein complexes due to Hochster \cite[Proposition 5.5, Theorem 6.7]{Ho} (see also \cite[Theorem 5.1]{S}):

\begin{lemma}\label{lm31}
Let $K$ be a field. Let $\Delta$ be a simplicial complex and $\Gamma:=\core \Delta$. Denote by $\widetilde{\chi}(\Gamma)$ the reduced Euler characteristic of $\Gamma$. Then $\Delta$ is Gorenstein over $K$ if and only if one of the following conditions holds:
\begin{enumerate}
 \item
 $\Delta=\emptyset$;
 \item
 $\Delta$ consists of one or two vertices;
 \item
 $\Delta$ is Cohen--Macaulay over $K$ of dimension $\geq1$, $\widetilde{\chi}(\Gamma)=(-1)^{\dim \Gamma}$, and for any face $F$ with $\dim\link_\Delta F=1$, $\link_\Delta F$ is either an $n$-gon ($n\geq 3$), or a path with at most 3 vertices.
\end{enumerate}
\end{lemma}

We will also need the following characterization of (locally) complete intersection complexes due to Terai and Yoshida \cite[Corollary 1.10, Proposition 1.11]{TY}:

\begin{lemma}\label{lm32}
 Let $K$ be a field. Let $\Delta$ be a simplicial complex of dimension $\geq1$. Assume that $K[\Delta]$ satisfies Serre's condition $(S_2)$.
 \begin{enumerate}[\rm(a)]
 \item If $\dim \Delta=1$, then the following conditions are equivalent:
  \begin{enumerate}[\rm(i)]
 \item
 $\Delta$ is a locally complete intersection complex;
 \item
 $\Delta$ is a locally Gorenstein complex;
 \item
 $\Delta$ is either an $n$-gon ($n\geq 3$), or an $m$-vertex path ($m\geq 2$).
 \end{enumerate}

 \item If $\dim \Delta\geq2$, then the following conditions are equivalent:
 \begin{enumerate}[\rm(i)]
 \item
 $\Delta$ is a complete intersection complex;
 \item
 $\Delta$ is a locally complete intersection complex;
 \item
 for any face $F$ with $\dim\link_\Delta F=1$, $\link_\Delta F$ is a complete intersection complex.
 \end{enumerate}
  \end{enumerate}
\end{lemma}

Recall that an arbitrary Noetherian ring $R$ is said to \emph{satisfy Serre's condition $(S_2)$} if $\depth R_P\geq\min\{\dim R_P,2\}$ for every prime ideal $P$ of $R$. So for example, all Cohen--Macaulay rings satisfy $(S_2)$. For an ideal $I$ in a polynomial ring $S$, we also say that $I$ satisfies $(S_2)$ when so does the quotient ring $S/I$.

From the above results one might see that the main part in the proof of Theorem \ref{th11} is the implication (iii)$\Rightarrow$(iv). For this, some more preparations are needed.

\begin{lemma}\label{lm33}
 Let $\Delta$ and $\Sigma$ be respectively the broken circuit complex and the matroid complex of an ordered matroid $(M,<)$. Let $F$ be a face of $\Delta$. Denote by $V_1$ and $V_2$ the vertex sets of $\link_\Delta F$ and $\link_\Sigma F$, respectively. Assume $V_2-V_1\ne\emptyset.$ Then there exists $u\in V_1$ such that $u< v$ for all $v\in V_2-V_1$.
\end{lemma}

\begin{proof}
 As $\Delta\subseteq \Sigma$, it is obvious that $V_1\subseteq V_2$. Let  $v=\min(V_2-V_1)$. Since $v\not\in V_1$, $F\cup\{v\}\not\in\Delta$. Thus there exists a circuit $C$ of $M$ such that $bc(C)\subseteq F\cup\{v\}$. We have $v\in bc(C)$, since otherwise $bc(C)\subseteq F$, contradicting the fact that $F\in \Delta$. Hence $u:=\min C< v.$ We will show that $u\in V_1.$ Indeed, since $v\in V_2$, $F\cup\{v\}\in\Sigma$, i.e., $F\cup\{v\}$ is an independent set of $M$. Therefore, $C$ is the unique circuit contained in $F\cup\{v,u\}$; see \cite[Proposition 1.1.6]{Ox}. It follows that $F\cup\{u\}$ is an independent set of $M$, or in other words, $u\in V_2.$ Note that $u\not\in V_2-V_1$ (since $u< v$), we obtain $u\in V_1$, as desired.
\end{proof}

\begin{lemma}\label{lm34}
Let $\Delta$ be the broken circuit complex of an ordered matroid $(M,<)$. Assume $\dim \Delta\geq1$. If $F$ is a face of $\Delta$ with $\link_\Delta F$ an $n$-gon, then $n\leq 4$.
\end{lemma}

\begin{proof}
Suppose on the contrary that there exists a face $F\in\Delta$ such that $\link_\Delta F$ is an $n$-gon $v_1v_2\ldots v_n$ with $n\geq5$ ($\{v_1,v_2\},\{v_2,v_3\},\ldots,\{v_n,v_1\}$ are the edges of $v_1v_2\ldots v_n$). Assume $v_1=\min\{v_i:i=1,\ldots,n\}$. Since $B_1=F\cup\{v_1,v_2\}$ and $B_2=F\cup\{v_3,v_4\}$ are bases of $M$, either $B_3=F\cup\{v_1,v_3\}$ or $B_4=F\cup\{v_1,v_4\}$ is a basis of $M$; see, e.g., \cite[Corollary 1.2.5]{Ox}. Note that neither $B_3$ nor $B_4$ belongs to $\Delta$ as $n\geq5$. So in the following argument we may assume without loss of generality that $B_3$ is a basis of $M$ and $B_3\not\in\Delta$. Denote by $\Sigma$ the matroid complex of $M$. Then $v_1\in \link_\Sigma(F\cup\{v_3\})-\link_\Delta(F\cup\{v_3\})$. By Lemma \ref{lm33}, there exists $u\in\link_\Delta(F\cup\{v_3\})$ with $u< v_1$. In particular, $u\in\link_\Delta
F$ and $u< v_1$. But this is impossible because it contradicts our assumption on $v_1$.
\end{proof}

\begin{lemma}\label{lm35}
Let $\Delta$ be the broken circuit complex of an ordered matroid $(M,<)$. If $\Delta$ is an $m$-vertex path, then $m\leq 3$.
\end{lemma}

\begin{proof}
Recall that $\Delta$ is a cone over the reduced broken circuit complex $\overline{BC}(M)$ which consists of vertices since $\dim \Delta=1$. It follows that if $\Delta$ is a path then $\overline{BC}(M)$ has at most 2 vertices, and hence, $\Delta$ has at most 3 vertices.
\end{proof}

We are now ready to prove Theorem \ref{th11}.

\begin{proof}[Proof of Theorem \ref{th11}]
The case $\Delta=\emptyset$ is trivial. The case $\dim \Delta=0$ is also trivial, because in this case $\Delta$ consists of only one vertex (since $\Delta$ is a cone over the empty complex $\overline{BC}(M)$).

If $\dim \Delta= 1$, then $F=\emptyset$ is the only face of $\Delta$ with $\dim\link_\Delta F=1$, and $\link_\Delta \emptyset=\Delta$. So we have (i)$\Leftrightarrow$(iii) and (iv)$\Leftrightarrow$(vi). Moreover, the equivalence  (ii)$\Leftrightarrow$(v) follows from Lemma \ref{lm32}, and the implication (i)$\Rightarrow$(ii) is well-known; see \cite[Proposition 5.6]{Ho}. Thus it remains to show (ii)$\Rightarrow$(vi). By Lemma \ref{lm32}, $\Delta$ is either an $n$-gon ($n\geq 3$) or an $m$-vertex path ($m\geq2$). Now by Lemma \ref{lm34} and Lemma \ref{lm35}, $n\leq4$ and $m\leq3$. In any case, it is easy to see that $\Delta$ is a complete intersection.

Now assume that $\dim \Delta\geq2$. Then we have (iv)$\Leftrightarrow$(v)$\Leftrightarrow$(vi) from Lemma \ref{lm32} and (vi)$\Rightarrow$(i)$\Rightarrow$(ii)$\Rightarrow$(iii) as it is well-known; see again \cite[Proposition 5.6]{Ho}. Finally, to prove (iii)$\Rightarrow$(iv), let $F$ be a face of $\Delta$ with $\dim\link_\Delta F=1$. Then by Lemma \ref{lm31}, $\link_\Delta F$ is either an $n$-gon or a path with at most 3 vertices. Using Lemma \ref{lm34}, one gets that $n\leq4$ when $\link_\Delta F$ is an $n$-gon. So we come to the same situation as in the case $\dim \Delta= 1$ above, in which $\link_\Delta F$ is easily seen to be a complete intersection.
\end{proof}

A classical theorem due to Cowsik and Nori \cite{CN} implies that a simplicial complex $\Delta$ is a complete intersection if and only if all powers of its Stanley--Reisner ideal $I_\Delta^m$ are Cohen--Macaulay. In \cite{TT}, Terai and Trung give a refinement for this result: they show that $\Delta$ is a complete intersection if and only if $I_\Delta^m$ is Cohen--Macaulay if and only if $I_\Delta^m$ satisfies Serre's condition $(S_2)$ for some $m\geq3$. They also point out that there are simplicial complexes $\Delta$, e.g., the $5$-gon, for which $I_\Delta^2$ is Cohen--Macaulay but $I_\Delta^m$ is not Cohen--Macaulay for every $m\geq3$. As one might have seen from above, this is not the case for the broken circuit complex.

\begin{proposition}\label{pr36}
Let $\Delta$ be the broken circuit complex of an ordered matroid $(M,<)$. Let $K$ be an arbitrary field. Denote by $I_\Delta$ the Stanley--Reisner ideal of $\Delta$ over $K$. Then the following conditions are equivalent:
\begin{enumerate}
\item
$\Delta$ is a complete intersection complex;
\item
$I_\Delta^m$ is Cohen--Macaulay for every $m\geq1$;
\item
$I_\Delta^m$ is Cohen--Macaulay for some $m\geq2$;
\item
$I_\Delta^m$ satisfies $(S_2)$ for some $m\geq2$.
\end{enumerate}
\end{proposition}

\begin{proof}
As in the proof of Theorem \ref{th11}, the cases $\Delta=\emptyset$ or $\dim\Delta=0$ are trivial. Let $\dim\Delta\geq1$. By \cite[Theorem 4.3]{TT}, it suffices to show that $\Delta$ is a complete intersection when $I_\Delta^2$ satisfies $(S_2)$. If $\dim\Delta=1$, then $(S_2)$ means that $I_\Delta^2$ is Cohen--Macaulay. Hence by \cite[Corollary 3.4]{MT}, either $\Delta$ has at most 3 vertices, or $\Delta$ is a 4-gon or 5-gon. But by Lemma \ref{lm34}, $\Delta$ cannot be a 5-gon. For the remaining possibilities of $\Delta$, one checks easily that $\Delta$ is a complete intersection. Now assume $\dim\Delta\geq2$. Let $F$ be a face of $\Delta$ with $\dim\link_\Delta F=1$. Since $I_\Delta^2$ satisfies $(S_2)$, $I_{\link_\Delta F}^2$ also satisfies $(S_2)$; see \cite[Corollary 4.2]{TT}. Note that $\link_\Delta F$ cannot be a 5-gon by Lemma \ref{lm34}. So one may use the same argument as in the case $\dim\Delta=1$ to get that $\link_\Delta F$ is a complete intersection. Now by Lemma \ref{lm32}, $\Delta$ is a
complete intersection.
\end{proof}

The equivalence of conditions (i) and (ii) in the following consequence was proved in \cite[Theorem 4.4.10]{Sto}.

\begin{corollary}\label{co37}
 Let $\Sigma$ be the matroid complex of a matroid $M$. Let $K$ be an arbitrary field. Denote by $I_\Sigma$ the Stanley--Reisner ideal of $\Sigma$ over $K$. Then the following conditions are equivalent:
 \begin{enumerate}
 \item
 $\Sigma$ is Gorenstein over $K$;
  \item
 $\Sigma$ is a complete intersection complex;
 \item
$I_\Sigma^m$ is Cohen--Macaulay for every $m\geq1$;
\item
$I_\Sigma^m$ is Cohen--Macaulay for some $m\geq2$;
\item
$I_\Sigma^m$ satisfies $(S_2)$ for some $m\geq2$.
 \end{enumerate}
Moreover, if $\dim\Sigma\geq1$, then each of the above conditions is equivalent to any one of the following:
\begin{enumerate}[\rm(vi)]
  \item
 $\Sigma$ is locally Gorenstein over $K$;
 \item
 for any face $F$ with $\dim\link_\Sigma F=1$, $\link_\Sigma F$ is Gorenstein over $K$;
 \item
 for any face $F$ with $\dim\link_\Sigma F=1$, $\link_\Sigma F$ is a complete intersection complex;
  \item[\rm(ix)]
 $\Sigma$ is a locally complete intersection complex.
\end{enumerate}
\end{corollary}

\begin{proof}
 Let $\tilde{M}$ be the free dual extension of $M$. Then $\Sigma$ is isomorphic to the reduced broken circuit complex of $\tilde{M}$; see \cite[Theorem 4.2]{Br}. So we may consider the broken circuit complex $\tilde{\Delta}$ of $\tilde{M}$ as a cone over $\Sigma$. It follows that $\Sigma$ is Gorenstein (resp. a complete intersection) if and only if $\tilde{\Delta}$ is. Moreover, if $I_\Sigma^m$ satisfies $(S_2)$ then so does $I_{\tilde{\Delta}}^m$. Therefore, from Theorem \ref{th11} and Proposition \ref{pr36} we get the equivalence of conditions (i)-(v).

Observe that $\link_\Sigma F=\link_{\tilde{\Delta}}(F\cup e_0)$ for all $F\in \Sigma$, where $e_0$ is the apex of $\tilde{\Delta}$ over $\Sigma$. So if $\dim \Sigma\geq2$, then the implication (vii)$\Rightarrow$(viii) can be proved as in the proof of Theorem \ref{th11}. From this the equivalence of conditions (i) and (ii) with (vi)-(ix) follows easily. Now assume that $\dim \Sigma=1$. The only thing we need to prove here is the implication (vi)$\Rightarrow$(i). Similar to the proof of Theorem \ref{th11}, this will be done once we have shown that $m\leq3$ when $\Sigma$ is an $m$-vertex path. Suppose $\Sigma$ is a path containing 4 consecutive vertices $v_1,v_2,v_3,v_4$ in that order. Then $\{v_1,v_2\}$ and $\{v_3,v_4\}$ are bases of $M$. It follows that either $\{v_1,v_3\}$ or $\{v_1,v_4\}$ is a basis of $M$; see \cite[Corollary 1.2.5]{Ox}. Thus $\{v_1,v_3\}\in \Sigma$ or $\{v_1,v_4\}\in \Sigma$, and hence, $\Sigma$ cannot be a path. This contradiction concludes the proof.
\end{proof}

\begin{remark}
 In Corollary \ref{co37}, without the assumption on the dimension of the matroid complex, conditions (i)-(v) may not be equivalent to conditions (vi)-(ix). For instance, let $M$ be the uniform matroid $U_{1,n}$ with $n\geq3$. Then $\Sigma$ is a 0-dimensional complex consists of $n$ vertices. So $\Sigma$ is locally Gorenstein (and a locally complete intersection), but not Gorenstein by Lemma \ref{lm31}(ii). Of course in this case, conditions (vii), (viii) in Corollary \ref{co37} hold vacuously.
\end{remark}

\section{The Gorensteinness of the Orlik--Terao algebra}

We present the proofs of Theorem \ref{th12} and Theorem \ref{th13} in this section. For this, the following simple lemma will play an important role.

\begin{lemma}\label{lm41}
 Let $M_1,M_2,M_3$ be simple matroids such that $M_1$ is either a direct sum or a parallel connection of $M_2$ and $M_3$. Let
$(h_{i,0},\ldots,h_{i,s_i})$ denote the $h$-vector of a broken circuit complex of $M_i$, where $h_{i,s_i}\ne 0$. Then the following conditions are equivalent:
\begin{enumerate}
 \item $h_{1,0}=h_{1,s_1}$ and $h_{1,1}=h_{1,s_1-1}$;
 \item $h_{i,0}=h_{i,s_i}$ and $h_{i,1}=h_{i,s_i-1}$ for $i=2,3$.
\end{enumerate}
\end{lemma}

\begin{proof}
 Let $r_i=\rank(M_i)$ and denote by $h_{M_i}(t)=\sum_{j=0}^{s_i}h_{i,j}t^{r_i-j}$ the $h$-polynomial of the broken circuit complex of $M_i$. Then by Proposition \ref{pr21}(ii) and Proposition \ref{pr22}(iii), we have
$$h_{M_1}(t)=
\begin{cases}
h_{M_2}(t)h_{M_3}(t) &\text{if  } M_1=M_2\oplus M_3,\\
t^{-1}h_{M_2}(t)h_{M_3}(t) &\text{if  } M_1=P(M_2,M_3).
\end{cases}
$$
It follows that in both cases,
$$\begin{aligned}
    h_{1,0}&=h_{2,0}h_{3,0}, &h_{1,1}&=h_{2,0}h_{3,1}+h_{2,1}h_{3,0},\\
   h_{1,s_1}&=h_{2,s_2}h_{3,s_3}, &h_{1,s_1-1}&=h_{2,s_2}h_{3,s_3-1}+h_{2,s_2-1}h_{3,s_3}.
 \end{aligned}$$
Thus, the implication (ii)$\Rightarrow$(i) is obvious. For the converse, we first have $h_{i,s_i}=h_{i,0}=1$ for $i=2,3$ because $h_{1,s_1}=h_{1,0}=1$. Let $N$ be a connected component of $M_2$. Then $h_N(t)$, the $h$-polynomial of the broken circuit complex of $N$, divides $h_{M_2}(t)$ by Proposition \ref{pr21}(ii). It follows that $\beta(N)$, which is the last nonzero coefficient of $h_N(t)$ by Proposition \ref{pr21}(i) and (iii), must divide $h_{2,s_2}=1$. Hence, $\beta(N)=1$. From Proposition \ref{pr23}, this means that $N$ is either a coloop or a series-parallel network. Thus every connected component of $M_2$ is a graphic matroid, which implies that $M_2$ itself is also graphic. The same conclusion holds for $M_3$. In particular, $M_2$ and $M_3$ are representable; see \cite[Proposition 5.1.2]{Ox}. Now applying Proposition \ref{pr21}(v) we get $h_{i,1}\leq h_{i,s_i-1}$ for $i=2,3$. The equalities are now forced by the relation $h_{1,1}=h_{1,s_1-1}$.
\end{proof}

\begin{proof}[Proof of Theorem \ref{th12}]
 (i)$\Rightarrow$(ii)$\Rightarrow$(iii) is trivial.

(iii)$\Rightarrow$(iv): By Proposition \ref{pr22}(i), (v) and Lemma \ref{lm41} we may assume that $M$ is a simple, connected, and parallel irreducible. Let $r=\rank M$. If $r=1$, then $M$ is a coloop. Suppose $r\geq 2$. We prove by induction on $r$ that $M\cong U_{r,r+1}$.

If $r=2$, then the $h$-polynomial of $M$ is $h_M(t)=t^2+t$. This can be the case only when $M\cong U_{2,3}$; see \cite[Theorem 3.1(1b)]{Br3}. Assume now that $r\geq 3$. Since $M$ is connected, $h_{r-1}\ne 0$ by Proposition \ref{pr21}(iii). So by assumption, $\beta(M)=h_{r-1}=h_0=1$. It follows from Proposition \ref{pr23} that $M=M(G)$, where $G$ is a simple, $2$-connected graphical series-parallel network. Note that $G$ has a vertex $v$ of degree 2. Let $e_1,e_2$ be the two incident edges of $v$. By Proposition \ref{pr21}(iv),
\begin{equation}\label{eq41}
h_{M}(t)=h_{M-e_1}(t)+h_{\overline{M/e_1}}(t),
\end{equation}
where $\overline{M/e_1}$ is the simplification of $M/e_1.$ Since $\rank(M-e_1)=r$ and $\rank(\overline{M/e_1})=r-1$ \cite[Proposition 2.4]{Br}, we may write
$$\begin{aligned}
h_{M-e_1}(t)&=h'_0t^{r}+h'_1t^{r-1}+\cdots+h'_{r-1}t,\\
h_{\overline{M/e_1}}(t)&=h''_0t^{r-1}+h''_1t^{r-2}+\cdots+h''_{r-2}t.
\end{aligned}$$
We will show that the matroid $\overline{M/e_1}$ satisfies the induction hypothesis. First of all, $\overline{M/e_1}$ is simple by definition. The remaining conditions will be shown through the following claims.

\begin{claim}\label{cl1}
$\overline{M/e_1}$ is connected.
\end{claim}
This follows from Proposition \ref{pr22}(iv) and the assumption that $M$ is parallel irreducible.

\begin{claim}\label{cl2}
$h''_{r-2}=h''_0=1$.
\end{claim}
Since $\overline{M/e_1}$ is connected, $h''_{r-2}\ne 0$. On the other hand, $h''_{r-2}\leq h_{r-1}=1$ by \eqref{eq41}. Hence, $h''_{r-2}=1$.

\begin{claim}\label{cl3}
$h''_1= h''_{r-3}.$
\end{claim}
We have $h'_{r-1}=h_{r-1}-h''_{r-2}=0$, thus $M-e_1$ is separable by Proposition \ref{pr21}(iii). It then follows from Proposition \ref{pr22}(iv) that $M$ is a series connection relative to $e_1$. Hence $M$ has at most one 3-element circuit containing $e_1$ by Proposition \ref{pr22}(ii). This implies that $|\overline{M/e_1}|\geq |M/e_1|-1=|M|-2$ since the contraction $M/e_1$ contains at most one 2-element circuit. Now by Proposition \ref{pr21}(i),
$$ h''_1=|\overline{M/e_1}|-\rank(\overline{M/e_1})\geq (|M|-2)-(r-1)=h_1-1.$$
 On the other hand, from Proposition \ref{pr21}(v) we have $h''_0+h''_1\leq h''_{r-2}+ h''_{r-3}$, which implies $h''_1\leq  h''_{r-3}$ by Claim \ref{cl2}. Moreover, it follows from \eqref{eq41} that $h''_{r-3}\leq h''_{r-3}+h'_{r-2}= h_{r-2}$. Thus we have shown
$$h_1-1\leq  h''_1\leq h''_{r-3}\leq h_{r-2}=h_1.$$
Therefore, the equality $h''_1= h''_{r-3}$ will follow if we prove that $h''_1\ne h_1-1$. Indeed, if $h''_1= h_1-1$ then $M/e_1$ contains one 2-element circuit. This means that $M$ has one 3-element circuit $C$ which contains $e_1$. Note that every circuit of $M$ containing $e_1$ must also contain $e_2$ since the vertex $v$ has degree 2. So $C=\{e_1,e_2,e\}$ with $e\in M$. As $r\geq 3$, $M$ properly contains $C$. Now one may easily check that $M$ is the parallel connection of $M-\{e_1,e_2\}$ and $C$ relative to $e$. But this contradicts the assumption that $M$ is parallel irreducible.

 \begin{claim}
 $\overline{M/e_1}$ is parallel irreducible.
 \end{claim}
It follows from the proof of Claim \ref{cl3} that $\overline{M/e_1}=M/e_1$. Thus by Proposition \ref{pr22}(iv), we need to prove that $\overline{(M/e_1)/e}$ is connected for all $e\in M/e_1$. It then suffices to show that $M/e_1-e$ is separable for all $e\in M/e_1$, by Proposition \ref{pr21}(iv) and Claim \ref{cl1}. Suppose there exists $e\in M/e_1$ with $M/e_1-e$ connected. We have $M/e_1-e=(M-e)/e_1$; see \cite[Proposition 3.1.26]{Ox}. As $M$ is parallel irreducible, $\overline{M/e}$ is connected by Proposition \ref{pr22}(iv). So by Proposition \ref{pr21}(iv), $M-e$ is separable (recall that $\beta(M)=1$). Now since $(M-e)/e_1$ is connected, we must have $M-e=\{e_1\}\oplus M'$ with $M'$ connected. It follows that $M-\{e_1,e\}$ is connected, and hence $M-e_1$ has at most 2 connected components. We will show that this is impossible. Indeed, one sees from the proof of Claim \ref{cl3} (i.e. $h_1''=h_1$) that $h''_{r-3}=h_{r-2}$. This implies $h'_{r-2}=h_{r-2}-h''_{r-3}=0$. So by Proposition \ref{pr21}(iii), $M-
e_1$ must have at least 3 connected components, a contradiction.

Now we may use induction hypothesis to conclude that $M/e_1\cong U_{r-1,r}$. Since $M$ is connected, it follows that $M\cong U_{r,r+1}$.

(iv)$\Rightarrow$(i): We may assume that $M$ is connected. The cases $M$ is a coloop or $M\cong U_{m,m+1}$ are trivial. Now assume that $M=P(M',U_{m,m+1})$ is the parallel connection of two matroids $M'$ and $U_{m,m+1}$ relative to their only common point $e$, in which $M'$ admits an order $<$ on its ground set such that $BC(M',<)$ is a complete intersection. Then we have the following description for the set of circuits of $M$:
\begin{equation}\label{eq42}
\mathcal{C}(M)=\mathcal{C}(M')\cup\{C_{m+1}\}\cup \{C\cup C_{m+1}-e:C\in\mathcal{C}(M')\},
\end{equation}
 where $C_{m+1}$ is the unique circuit of the matroid $U_{m,m+1}$. Since $BC(M',<)$ is a complete intersection, the set $\mathbf{m}(BC(M',<))$ of minimal broken circuits of $(M',<)$ consists of disjoint elements; see \cite[Theorem 4.1]{LR}. Now let $\prec$ be an extension of $<$ to the ground set of $M$ such that $e=\min_\prec(U_{m,m+1})$. Then it can be easily deduced from \eqref{eq42} that $\mathbf{m}(BC(M,\prec))=\mathbf{m}(BC(M',<))\cup\{C_{m+1}-e\}$. Hence $\mathbf{m}(BC(M,\prec))$ also consists of disjoint elements, from which it follows that the broken circuit complex $BC(M,\prec)$ is a complete intersection, again by \cite[Theorem 4.1]{LR}.
\end{proof}

Finally, we prove Theorem \ref{th13}.

\begin{proof}[Proof of Theorem \ref{th13}]
(i)$\Rightarrow$(ii) is trivial. (iii)$\Rightarrow$(i) is also trivial because for any ordering $<$ of $\mathcal{A}$ the Stanley--Reisner ideal of the broken circuit complex $BC(M(\mathcal{A}),<)$ is an initial ideal of the Orlik--Terao ideal of $\mathcal{A}$; see \cite[Theorem 4]{PS}.

(ii)$\Rightarrow$(iii): When $\mathbf{C}(\mathcal{A})$ is Gorenstein, it is well-known that its $h$-vector is symmetric; see, e.g., \cite[p. 51]{S}. Since this is also the $h$-vector of broken circuit complexes of $M(\mathcal{A})$, Theorem \ref{th12} applies.
\end{proof}


\begin{thebibliography}{99}

\bibitem{B}
A. Bj\"{o}rner, \emph{The homology and shellability of matroids and geometric lattices}. In \emph{Matroid Applications}, 226--283, Encyclopedia Math. Appl. {\bf 40}, Cambridge University Press, Cambridge, 1992.

\bibitem{BZ}
A. Bj\"{o}rner and G. Ziegler, \emph{Broken circuit complexes: factorizations and generalizations}. J. Combin. Theory, Series B {\bf 51} (1991), no. 1, 96--126.

\bibitem{Br2}
T. Brylawski, \emph{A combinatorial model for series-parallel networks}. Trans. Amer. Math. Soc. {\bf 154} (1971), 1--22.

\bibitem{Br}
T. Brylawski, \emph{The broken-circuit complex}. Trans. Amer. Math. Soc. {\bf 234} (1977), 417--433.

\bibitem{Br3}
T. Brylawski, \emph{Connected matroids with the smallest Whitney numbers}.
Discrete Math. {\bf 18} (1977), no. 3, 243--252.

\bibitem{BrOx}
T. Brylawski and J. Oxley,
\emph{The broken-circuit complex: its structure and factorizations}.
European J. Combin. {\bf 2} (1981), no. 2, 107--121.

\bibitem{BrOx2}
T. Brylawski and J. Oxley,
\emph{The Tutte polynomial and its applications}. In \emph{Matroid Applications}, 123--225, Encyclopedia Math. Appl. {\bf 40}, Cambridge University Press, Cambridge, 1992.

\bibitem{CN}
R. C. Cowsik and M. V. Nori, \emph{On the fibres of blowing up}.
J. Indian Math. Soc. (N.S.) {\bf 40} (1976), no. 1-4, 217--222.

\bibitem{Cr}
H. Crapo, \emph{A higher invariant for matroids}. J. Combinatorial Theory {\bf 2} (1967), 406--417.

\bibitem{DGT}
G. Denham, M. Garrousian and  S. Tohaneanu, \emph{Modular decomposition of the Orlik--Terao algebra of a hyperplane arrangement}. To appear in Annals of Combinatorics.

\bibitem{Di}
G. A. Dirac, \emph{A property of 4-chromatic graphs and some remarks on critical graphs}.
J. London Math. Soc. {\bf 27} (1952), 85--92.


\bibitem{Ho}
M. Hochster, \emph{Cohen--Macaulay rings, combinatorics, and simplicial complexes}. In \emph{Ring Theory II} (Proc. Second Oklahoma Conference), 171--223, Dekker, New York, 1977.

\bibitem{LR}
D. V. Le and T. R\"{o}mer, \emph{Broken circuit complexes and hyperplane arrangements}. J. Algebraic Combin. {\bf 38} (2013), no. 4, 989--1016.

\bibitem{MT}
N. C. Minh and N. V. Trung, \emph{Cohen--Macaulayness of powers of two-dimensional squarefree monomial ideals}.
J. Algebra {\bf 322} (2009), no. 12, 4219--4227.

\bibitem{OT}
P. Orlik and  H. Terao, \emph{Arrangements of Hyperplanes}. Grundlehren Math. Wiss.,
Bd. {\bf 300}, Springer-Verlag, Berlin-Heidelberg-New York, 1992.

\bibitem{OT2}
P. Orlik and  H. Terao, \emph{Commutative algebras for arrangements}. Nagoya Math.
J. {\bf 134} (1994), 65--73.

\bibitem{Ox}
J. Oxley, \emph{Matroid theory, Second Edition}. Oxford Graduate Texts in Mathematics, vol. {\bf 21}, Oxford University Press, Oxford, 2011.


\bibitem{PS}
N. Proudfoot and D. Speyer, \emph{A broken circuit ring}. Beitr\"{a}ge Algebra Geom. {\bf 47} (2006), 161--166.

\bibitem{Pr}
J. S. Provan, \emph{Decompositions, shellings, and diameters of simplicial complexes and convex polyhedra}. Thesis, Cornell University, Ithaca, NY, 1977.

\bibitem{Sc}
H. Schenck, \emph{Resonance varieties via blowups of $\mathbb{P}^2$ and scrolls},
Int. Math. Res. Not. {\bf 20} (2011), 4756--4778.

\bibitem{Sc2}
H. Schenck, \emph{Hyperplane Arrangements: Computations and Conjectures}.
Arrangements of hyperplanes - Sapporo 2009, 323--358,
Adv. Stud. Pure Math., {\bf 62}, Math. Soc. Japan, Tokyo, 2012.

\bibitem{ST}
H. Schenck and S. Tohaneanu, \emph{The Orlik--Terao algebra and $2$-formality}. Math.
Res. Lett. {\bf 16} (2009), 171--182.

\bibitem{S2}
R. P. Stanley, \emph{Hilbert functions of graded algebras}. Adv. Math. {\bf 28} (1978), 57--83.

\bibitem{S1}
R. P. Stanley, \emph{On the Hilbert function of a graded Cohen--Macaulay domain}. J. Pure Appl. Algebra {\bf 73} (1991), no. 3, 307--314.

\bibitem{S}
R. P. Stanley, \emph{Combinatorics and Commutative Algebra, Second Edition}. Progress in Mathematics, vol. {\bf 41}, Birkh\"{a}user, 1996.

\bibitem{Sto}
E. Stokes, \emph{The $h$-vectors of matroids and the arithmetic degree of squarefree strongly stable ideals}. Ph.D. thesis, University of Kentucky, 2008.

\bibitem{Sw}
E. Swartz, \emph{$g$-elements of matroid complexes}.
J. Combin. Theory Ser. B {\bf 88} (2003), no. 2, 369--375.

\bibitem{TT}
N. Terai and N. V. Trung, \emph{Cohen--Macaulayness of large powers of Stanley--Reisner ideals}.
Adv. Math. {\bf 229} (2012), no. 2, 711--730.

\bibitem{TY}
N. Terai and K. Yoshida,  \emph{Locally complete intersection Stanley--Reisner ideals}. Illinois J. Math. {\bf 53} (2009), 413--429.

\bibitem{T}
H. Terao, \emph{Algebras generated by reciprocals of linear forms}. J. Algebra {\bf 250} (2002), 549--558.

\bibitem{Wh}
H. Whitney, \emph{A logical expansion in mathematics}. Bull. Amer. Math. Soc. {\bf 38} (1932), 572--579.

 \bibitem{W}
H. Wilf, \emph{Which polynomials are chromatic?}. Proc. 1973 Rome International Colloq.
Combinatorial Theory I, pp. 247--257, Accademia Nazionale dei Lincei, Rome, 1976.

\bibitem{Za}
T. Zaslavsky, \emph{The M\"{o}bius function and the characteristic polynomial}. In \emph{Combinatorial geometries}, 114--138, Encyclopedia Math. Appl. {\bf 29}, Cambridge University Press, Cambridge, 1987.
\end{thebibliography}
\end{document}